\documentclass[11pt]{article}

\usepackage{caption}
\DeclareCaptionLabelSeparator{twospace}{. }
\usepackage{indentfirst}
\usepackage{mathrsfs}
\usepackage{enumerate}
\usepackage[a4paper,margin=1in]{geometry}  
\usepackage[ruled,linesnumbered,vlined]{algorithm2e}
\usepackage{amsmath,amsthm,tikz,float}

\usepackage[colorlinks=true,citecolor=black,linkcolor=black,urlcolor=blue]{hyperref}
\usepackage{mathrsfs}
\usepackage{amssymb}

\usepackage{amsfonts}
\usepackage[shortlabels]{enumitem}
\setenumerate[1]{itemsep=0pt,partopsep=0pt,parsep=\parskip,topsep=0pt}
\usepackage{amsmath,amsthm,tikz,float,amssymb}
\usepackage[T1]{fontenc}
\usepackage[utf8]{inputenc}
\usepackage{authblk}

\usepackage{setspace}

\newtheorem{theorem}{Theorem}[section]
\newtheorem{lemma}[theorem]{Lemma}
\newtheorem{definition}[theorem]{Definition}

\newtheorem{problem}[theorem]{Problem}
\newtheorem{corollary}[theorem]{Corollary}
\newtheorem*{remark}{Remark}

\newenvironment{MainThmproof1}
{\begin{proof}[{Proof of Theorem \ref{MainResult} }]}
{\end{proof}}

\newenvironment{MainThmproof2}
{\begin{proof}[{Proof of Theorem \ref{MainResult} }]}
{\end{proof}}

\tikzset{
  module/.style={draw,rounded corners,minimum width=1.3cm,minimum height=1cm},
  terminal/.style={circle,draw,fill=black,inner sep=2pt},
  edgelight/.style={gray!70}
}

\begin{document}


\title{On regular homogeneously traceable nonhamiltonian graphs}
\author{Hangdi Chen$^{1,}$\thanks{chenhangdi188@126.com} ,  Yaojun Chen$^{2,}$\thanks{yaojunc@nju.edu.cn} \\
    {\small $^1$Fujian Key Laboratory of Financial Information Processing, Putian University, }\\ 
    {\small Putian, 351100, P.R. China}\\
    {\small $^2$School of Mathematics, Nanjing University, Nanjing, 210093, P.R. China}
	
}
\date{}
\maketitle

\begin{abstract}
A graph is homogeneously traceable if each vertex is an endpoint of a Hamiltonian path. Chartrand, Gould, and Kapoor (1979) proved irregular homogeneously traceable nonhamiltonian graphs exist for every order $n\ge 9$. Hu and Zhan (DAM, 2022) considered the $3$-regular and $4$-regular cases and asked which order $n$ can be realized by a $k$-regular homogeneously traceable nonhamiltonian graph. Recently, Liu and Qiao (DAM, 2026) showed that $n=p(k-1)+q\ge 6(k-1)+q$ can be realized if $k\ge 5$ is odd and $q\in\{0,2,4,6\}$, or $k\ge 6$ is even and $q\in\{0,1,...,6\}$. In this paper, we show that for any $k\ge 6$ and $n\ge 6(k-2)$, there exists a $k$-regular homogeneously traceable nonhamiltonian graph of order $n$.


\end{abstract}



\section{Introduction}

In this paper, all graphs considered are finite and  simple. For a graph $G$, let $V(G)$ be the \emph{vertex set} of $G$ and  $E(G)$ be the \emph{edge set} of $G$.  The number of vertices of $G$ is called the \emph{order} of $G$, and is denoted by $v(G)=|V(G)|$. The number of edges of $G$ is called the \emph{size} of $G$, and is denoted by $e(G)=|E(G)|$.  Let $N_G(v)$ denote the neighborhood of $v$ consisting of all vertices adjacent to $v$. The \emph{degree} of a vertex $v$, denoted by $d_G(v)$, is the size of $N_G(v)$, and the \emph{minimum degree} of $G$ is denoted by $\delta (G)$.  A graph $G$ is \emph{$k$-regular} if every vertex of $G$ has degree $k$. When no confusion can occur, we will omit the subscript $G$.

A \emph{Hamilton path} (\emph{cycle}) of a graph $G$ is a path (cycle) containing every vertex of $G$ exactly once. A graph $G$ is \emph{traceable} (\emph{Hamiltonian}) if it has a Hamilton path (cycle). A graph $G$ is \emph{nonhamiltonian} if it has no Hamilton cycle. The following concept is introduced by Skupie$\acute{n}$  \cite{Sku1980,Sku1984} in 1975.

\begin{definition}
\emph{A graph $G$ is said to be} homogeneously traceable \emph{if every vertex of $G$ is an endpoint of a Hamilton path. }
\end{definition}

Chartrand, Gould and Kapoor \cite{Chart1979} showed that every homogeneously traceable graph of order $n$ is hamiltonian for each integer $n$ satisfying $3\le n\le 8$, whereas for every $n\ge 9$ there exists a homogeneously traceable graph of order $n$ that is nonhamiltonian. This result was later rediscovered in \cite{Carl2021}, where ``homogeneously traceable'' was called ``fully strung''. Apart from the Petersen graph, the homogeneously traceable nonhamiltonian graphs constructed in \cite{Carl2021,Chart1979} are all irregular. This naturally leads to the question of whether regular homogeneously traceable nonhamiltonian graphs exist.

In 2022, Hu and Zhan \cite{Hu2022Zhan} proved that, for every even integer $n\ge 10$, there exists a $3$-regular homogeneously traceable nonhamiltonian graph of order $n$. They also proved that, for every integer $n\ge 18$, there exists a $4$-regular homogeneously traceable nonhamiltonian graph of order $n$. In the same paper, they posed the following problem.

\begin{problem}\textup{(Hu and Zhan \cite{Hu2022Zhan})}.
Given an integer $k\ge 4$, determine the integers $n$ such that there exists a $k$-regular homogeneously traceable nonhamiltonian graph of order $n$.
\end{problem}

Recently, Liu and Qiao \cite{Liu2026Qiao} constructed many such graphs for all larger degrees. And they \cite{Liu2026Qiao} proved the following theorem.

\begin{theorem}\label{LiuThm2026}
\textup{(Liu and Qiao \cite{Liu2026Qiao})}.  The following statements hold.

(i) For every odd integer $k\ge 5$, integer $p\ge 6$ and integer $q\in\{0,2,4,6\}$, there exists a $k$-regular homogeneously traceable nonhamiltonian graph of order $p(k-1)+q$. 

(ii) For every even integer $k\ge 6$, integer $p\ge 6$ and integer $q\in\{0,1,\ldots,6\}$, there exists a $k$-regular homogeneously traceable nonhamiltonian graph of order $p(k-1)+q$. 
\end{theorem}

\begin{remark}
\emph{For every odd integer $k\ge 11$, the only even residue classes not covered by Theorem \ref{LiuThm2026} are $q=8,10,\ldots,k-3$. For every even integer $k\ge 10$, the only residue classes not covered by Theorem \ref{LiuThm2026} are $q=7,8,\ldots,k-2$. As an immediate consequence of Theorem \ref{LiuThm2026}, we obtain the following corollary.}
\end{remark}

\begin{corollary}\label{LiuCorol2026}
\textup{(Liu and Qiao \cite{Liu2026Qiao})}. The following statements hold.

(i) For every odd  $k\in \{5,7,9\}$ and every even $n\ge 6(k-1)$, there exists a $k$-regular homogeneously traceable nonhamiltonian graph of order $n$. 

(ii) For every even $k\in\{6,8\}$ and every $n\ge 6(k-1)$, there exists a $k$-regular homogeneously traceable nonhamiltonian graph of order $n$.
\end{corollary}

In this paper, we settle a more general situation: $k\ge 6$ and $n\ge 6(k-2)$.  Our main result is stated as follows.

\begin{theorem}\label{MainResult} For any $k\ge 6$ and $n\ge 6(k-2)$, there exists a $k$-regular homogeneously traceable nonhamiltonian graph of order $n$.




\end{theorem}

\section{Preliminaries}
In this section, we introduce some notation and auxiliary results that are used in the proof of Theorem \ref{MainResult}. Let $G$ be a graph. The \emph{complement} $\overline{G}$ of $G$ is the graph with vertex set $V(G)$ in which two distinct vertices are adjacent if and only if they are nonadjacent in $G$. For two disjoint subsets $A$ and $B$ of $V(G)$, let $G[A]$ denote the subgraph of $G$ induced by $A$, and let $G[A,B]$ denote the bipartite subgraph with vertex set $A\cup B$ and edge set
\[
\{ab\in E(G): a\in A,\ b\in B\}.
\]
For a vertex $v\in V(G)$, we write $G-v$ for the subgraph induced by $V(G)\setminus\{v\}$. A \emph{$v$-path} in $G$ is a path having $v$ as an endpoint. Let $P$ and $Q$ be two paths in $G$. If $u$ and $v$ are the endpoints of $P$, and $v$ and $t$ are the endpoints of $Q$, then we write $uPv$ and $vQt$ to indicate the corresponding orientations of $P$ and $Q$, respectively. We denote by $uPvQt$ the sequence obtained by concatenating $P$ and $Q$ at $v$; in particular, if $P$ and $Q$ meet only at $v$, then $uPvQt$ is a path.

\begin{definition}
\emph{A graph $G$ is said to be }doubly homogeneously traceable \emph{if for any vertex $v\in V(G)$, there are two Hamilton $v$-paths $P$ and $Q$ such that the two edges incident with $v$ on $P$ and $Q$ are distinct.}
\end{definition}

We use the blow-up operation introduced in \cite{Hu2022Zhan,Liu2026Qiao}. 
Let $K_n$ denote the \emph{complete graph} of order $n$. For a vertex $v$ of a graph
$G$, set $t=d_G(v)$. The operation of \emph{blowing up $v$ into $K_t$} consists
of replacing $v$ with a copy of $K_t$ and adding the edges of a matching between
this copy of $K_t$ and the neighborhood $N_G(v)$. The blow-up of a vertex of
degree $3$ into $K_3$ is illustrated in Figure \ref{Blowexam}. Liu and Qiao
\cite{Liu2026Qiao} proved the following blow-up lemma.

\begin{figure}[H]
\begin{minipage}[t]{0.5\textwidth}
    \centering
    \begin{tikzpicture}
        \coordinate (a1) at (5,6.5);
        \coordinate (a2) at (5,5.5);
		\coordinate (a3) at (4,4.5);
		\coordinate (a4) at (6,4.5);

        \foreach \i in {1,2,3,4}
        {
            
            \fill (a\i) circle[radius=3pt];
        }
    
        \draw  (a4)--(a2) -- (a3) ;
       \draw  (a1)--(a2) ;

                  \node at (4.6,5.5) {$v$};
               \node at (5,6.8) {$u_1$};
 \node at (3.6,4.5) {$u_2$};
 \node at (6.4,4.5) {$u_3$};       
    \end{tikzpicture}
\end{minipage}
\begin{minipage}[t]{0.4\textwidth}
    \centering
\begin{tikzpicture}
        \coordinate (a1) at (5,6.5);
        \coordinate (a2) at (5,5.75);
		\coordinate (a3) at (4,4.5);
		\coordinate (a4) at (6,4.5);
\coordinate (a5) at (4.5,5);
\coordinate (a6) at (5.5,5);

        \foreach \i in {1,2,3,4,5,6}
        {
            
            \fill (a\i) circle[radius=3pt];
        }
    
        \draw  (a3)--(a5) ;
       \draw  (a1)--(a2) ;
       \draw  (a4)--(a6) ;
  \draw  (a2)--(a5) --(a6) --(a2) ;

               \node at (5,6.8) {$u_1$};
 \node at (3.6,4.5) {$u_2$};
 \node at (6.4,4.5) {$u_3$};  

 \node at (4.7,5.75) {$v_1$};
 \node at (4.2,5) {$v_2$};
 \node at (5.8,5) {$v_3$};    
    \end{tikzpicture}
\end{minipage}
  \caption{Blowing up $v$ into $K_3$.}\label{Blowexam}
\end{figure}
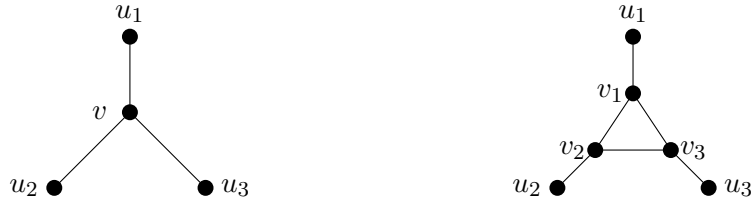

\begin{lemma}\label{BlowLem}
\textup{(Blow-up lemma \cite{Liu2026Qiao})}. Let $v$ be a vertex of degree $t$  $(t\ge3)$ in a doubly homogeneously traceable graph $G$. Suppose $H$ is the graph obtained from $G$ by blowing up $v$ into $K_t$. Then $H$ is also doubly homogeneously traceable.
\end{lemma}

A graph $G$ is said to be \emph{Hamilton-connected} if, for every pair of
distinct vertices $u,v\in V(G)$, there exists a Hamilton path in $G$ with
endpoints $u$ and $v$. We shall repeatedly use the following well-known theorem
of Erd\H{o}s and Gallai \cite{Erdos1959Gallai}.

\begin{theorem}\textup{(Erd\H{o}s and Gallai \cite{Erdos1959Gallai})}\label{ErdosThm}
Let $H$ be a graph of order $n\ge 3$, and $u,v$ are any pair distinct and nonadjacent vertices. If $d_H(u)+d_H(v)\ge n+1$, then $H$ is Hamilton-connected. 
\end{theorem}

Now, we give the definiton of the triangular terminal graph. 

\begin{definition}
\emph{
Let $U=\{u_1,u_2,u_3\}$ and $U'=\{u_1',u_2',u_3'\}$ be two disjoint vertex sets, each inducing a copy of $K_3$. Let $Q_1,Q_2,Q_3$ be pairwise vertex-disjoint internal graphs, all disjoint from $U\cup U'$. For each $i\in\{1,2,3\}$, place the internal graph $Q_i$ between $u_i$ and $u_i'$ (see Figure \ref{fig:terminal-frame}). Choose two nonempty subsets $X_i,Y_i\subseteq V(Q_i)$, called the }outer sets \emph{of $Q_i$, and define
\[
O_i=\{xu_i:x\in X_i\}
\quad\text{and}\quad
O_i'=\{yu_i':y\in Y_i\}.
\]
The only edges joining $Q_i$ to vertices outside $Q_i$ are those in $O_i\cup O_i'$. The resulting graph is called a }triangular terminal graph.
\end{definition}

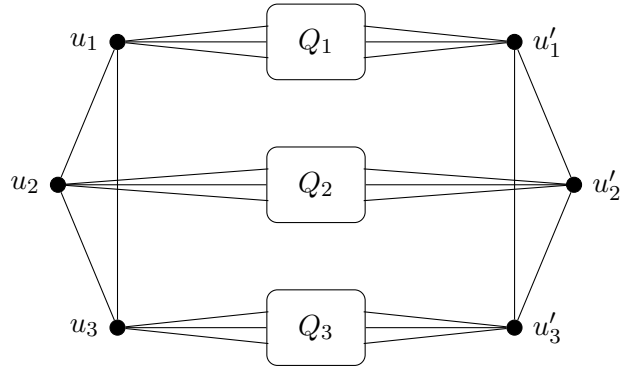
\begin{figure}[ht]
\centering
\begin{tikzpicture}[scale=1.05]
  \node[terminal,label=left:$u_1$] (u1) at (0,1.8) {};
  \node[terminal,label=left:$u_2$] (u2) at (-0.75,0) {};
  \node[terminal,label=left:$u_3$] (u3) at (0,-1.8) {};
  \node[terminal,label=right:$u'_1$] (v1) at (5,1.8) {};
  \node[terminal,label=right:$u'_2$] (v2) at (5.75,0) {};
  \node[terminal,label=right:$u'_3$] (v3) at (5,-1.8) {};
  \node[module] (q1) at (2.5,1.8) {$Q_1$};
  \node[module] (q2) at (2.5,0) {$Q_2$};
  \node[module] (q3) at (2.5,-1.8) {$Q_3$};
 
\coordinate (t1) at (1.9,2) ;
\coordinate (t2) at (1.9,0.2) ;
\coordinate (t3) at (1.9,-1.6) ;
\coordinate (x1) at (1.9,1.6) ;
\coordinate (x2) at (1.9,-0.2) ;
\coordinate (x3) at (1.9,-2) ;

\coordinate (y1) at (3.1,2) ;
\coordinate (y2) at (3.1,0.2) ;
\coordinate (y3) at (3.1,-1.6) ;
\coordinate (z1) at (3.1,1.6) ;
\coordinate (z2) at (3.1,-0.2) ;
\coordinate (z3) at (3.1,-2) ;
  \draw (u1)--(u2)--(u3)--(u1);
  \draw (v1)--(v2)--(v3)--(v1);
  \foreach \i in {1,2,3}{
    \draw (u\i)--(q\i);
\draw (u\i)--(t\i);
\draw (u\i)--(x\i);
    \draw (q\i)--(v\i);
\draw (y\i)--(v\i);
\draw (z\i)--(v\i);
  }
\end{tikzpicture}
\caption{The triangular terminal graph.}
\label{fig:terminal-frame}
\end{figure}

For an internal graph $Q$ placed between two vertices $u$ and $u'$ in a triangular terminal graph, and for distinct vertices
$x,y\in V(Q)\cup\{u,u'\}$, a \emph{spanning $(x,y)$-path through $Q$} is an $(x,y)$-path whose internal vertices lie in $V(Q)$ and that contains every vertex of $Q$. We say that $Q$ has the \emph{SLC property} if it satisfies the following three conditions.

\noindent \textbf{(A1)} There exist two spanning $(u,u')$-paths through $Q$ with distinct initial edges.

\noindent \textbf{(A2)} There exist two spanning $(u',u)$-paths through $Q$ with distinct initial edges.

\noindent \textbf{(A3)} For every $v\in V(Q)$, there exist two spanning $(v,u)$-paths through $Q$ with distinct initial edges.

Let $V_1=\{x_0,\ldots,x_{k-3}\}$ and $V_2=\{y_0,\ldots,y_{k-3}\}$ be vertex sets. And let $E_1=\{x_ix_j~|~i,j=0,1,\ldots,k-3,i\ne j\}$, $E_2=\{y_iy_j~|~i,j=0,1,\ldots,k-3,i\ne j\}$, $E_3=\{x_iy_i~|~i=0,1,\ldots,k-3\}$ and $E_4=\{x_iy_{i+1}~|~i=0,1,\ldots,k-3\}$ (where $i+1$ is taken  modulo $k-2$) be edge sets. For every integer $k\ge 4$, the ordinary internal graph $W_k$ introduced by
Liu and Qiao \cite{Liu2026Qiao} is the graph with vertex set $V(W_k)=V_1\cup V_2$ and edge set $E(W_k)=\cup_{i=1}^4 E_i$.
The base graph $G_k$ of Liu and Qiao \cite{Liu2026Qiao} is the triangular
terminal graph in which each internal graph is a copy of $W_k$, with outer
sets $V_1$ and $V_2$.

For every even integer $k\ge 6$, we introduce the following reconstruction of the graph $W_k$ introduced by Liu and Qiao \cite{Liu2026Qiao}. Starting from $W_k$, add a new vertex $z$, delete all edges in $E_4$, and add
the edges in $\bigcup_{j=1}^4 T_j$, where $T_1=\{x_iz~|~i=0,1,\ldots,\frac{k}{2}-1\}$, $T_2=\{y_iz~|~i=0,1,\ldots,\frac{k}{2}-1\}$, $T_3=\{x_iy_{i+1}~|~i=\frac{k}{2},\frac{k}{2}+1,\ldots,k-4\}$ and $T_4=\{x_{k-3}y_{\frac{k}{2}}\}$. The resulting graph is denoted by $W_k'$ and is called the reconstruction of $W_k$; its outer sets are $V_1$ and $V_2$. By construction,
$V(W_k')=V_1\cup V_2\cup\{z\}$ and $E(W_k')=(\cup_{i=1}^3 E_i)\cup (\cup_{j=1}^4 T_j)$.

Let $X$ and $Y$ be sets of vertices of a graph $G$. Denote by $E_G(X,Y)$ the set of edges of $G$ with one end in $X$ and the other end in $Y$. The set $E_G(X,V(G)\setminus X)$ is called the \emph{edge cut} of $G$ associated with $X$, and is denoted by $\partial (X)$. We will need the following two lemmas about the triangular terminal graph.

\begin{lemma}\label{DHTLem}
\textup{(Liu and Qiao \cite{Liu2026Qiao})}. If every internal graph  $Q_i$ in the triangular terminal graph $H$ is $SLC$, then $H$ is  doubly homogeneously traceable.
\end{lemma}

\begin{lemma}\label{NHLem}
If every internal graph $Q_i$ in the triangular terminal graph $H$ is nonempty, then $H$ is  nonhamiltonian.
\end{lemma}
\begin{proof}
Suppose, to the contrary, that $H$ contains a Hamilton cycle $F$, and assume that
$V(Q_j)\ne\emptyset$ for each $j\in\{1,2,3\}$. Fix
$i\in\{1,2,3\}$ and set $T_i=V(Q_i)$. We claim that $|E(F)\cap \partial_H(T_i)|=2.$ Indeed, since $F$ is a cycle and $T_i$ is a proper nonempty subset of $V(H)$, the number $|E(F)\cap \partial_H(T_i)|$ is positive and even. If $|E(F)\cap \partial_H(T_i)|\ne 2$, then $|E(F)\cap \partial_H(T_i)|\ge 4$. Since the only edges from
$Q_i$ to the rest of $H$ are incident with $u_i$ or $u_i'$, and since $d_F(u_i)=d_F(u_i')=2$, this would force both $u_i$ and
$u_i'$ to have their two incident edges in $F$ joining them to vertices of
$T_i$. Hence $F[T_i\cup\{u_i,u_i'\}]$ would contain a cycle component separated
from the rest of $F$, a contradiction. Therefore, $|E(F)\cap \partial_H(T_i)|=2.$
Moreover, suppose that\[
|E(F)\cap \partial_H(T_i)\cap E_H(T_i,\{u_i\})|=2
\quad\text{or}\quad
|E(F)\cap \partial_H(T_i)\cap E_H(T_i,\{u_i'\})|=2.\] Then, either $F[T_i\cup\{u_i\}]$ or
$F[T_i\cup\{u_i'\}]$ would contain a cycle component separated from the rest of
$F$, again a contradiction. It follows that
\[
|E(F)\cap \partial_H(T_i)\cap E_H(T_i,\{u_i\})|=1
\quad\text{and}\quad
|E(F)\cap \partial_H(T_i)\cap E_H(T_i,\{u_i'\})|=1.
\]
Thus, the subgraph of $F$ induced by $T_i\cup\{u_i,u_i'\}$ is a spanning
$(u_i,u_i')$-path through $Q_i$. Denote this path by $P_i$. Since $i$ was
arbitrary, such a path $P_i$ exists for every $i\in\{1,2,3\}$. 

Contracting each path $P_i$ to the edge $u_iu_i'$, we obtain a Hamilton cycle $F'$ in the triangular prism on $U\cup U'$ (see Figure \ref{fig:prism-obstruction}). Moreover, $u_iu_i'\in E(F')$ for each $i\in\{1,2,3\}$. It is easy to check that it is impossible. So $H$ is nonhamiltonian. 
\end{proof}

\begin{figure}[ht]
\centering
\begin{tikzpicture}[scale=1.05]
  \node[terminal,label=left:$u_1$] (u1) at (0,1.8) {};
  \node[terminal,label=left:$u_2$] (u2) at (-0.75,0) {};
  \node[terminal,label=left:$u_3$] (u3) at (0,-1.8) {};
  \node[terminal,label=right:$u'_1$] (v1) at (4,1.8) {};
  \node[terminal,label=right:$u'_2$] (v2) at (4.75,0) {};
  \node[terminal,label=right:$u'_3$] (v3) at (4,-1.8) {};
  \draw (u1)--(u2)--(u3)--(u1);
  \draw (v1)--(v2)--(v3)--(v1);
\foreach \i in {1,2,3}{\draw(u\i)--(v\i);}

\end{tikzpicture}
\caption{The triangular prism on $U\cup U'$.}
\label{fig:prism-obstruction}
\end{figure}
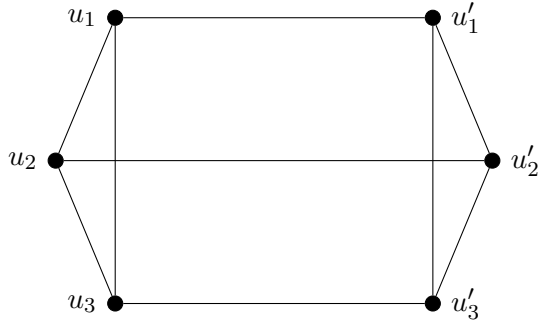

\section{The odd-degree case of Theorem \ref{MainResult}}

In Subsection \ref{Odd11}, we settle all residue classes not covered by Theorem \ref{LiuThm2026} for every odd degree $k$. In Subsection \ref{Odd7}, we reduce the order threshold in Corollary \ref{LiuCorol2026} to $6(k-2)$ for every odd degree $k$.

\subsection{Odd degrees $k\ge 11$: all even orders $n\ge 6(k-1)$}\label{Odd11}
In this subsection, we first introduce an internal graph $D_{k,r}$ to settle the remaining residue classes for every odd integer $k\ge 11$. Let $t=\lfloor\frac{k}{2}\rfloor$ and choose an integer $4\le r\le t-1$. Let $L=k-2+r=2t-1+r$. Let $(V_1,V_2)$  be a vertex partition of $V(D_{k,r})$, where $V_1=\{x_0,\ldots,x_{L-1}\}$ and $V_2=\{y_0,\ldots,y_{L-1}\}$. Let $E(D_{k,r})=\cup_{j=1}^8 M_j$, where

$M_1=\{x_ix_{i+ j}~|~i=0,1,\ldots,L-1;j=1,2,\ldots,t-1\}$, 

$M_2=\{x_ix_{i- j}~|~i=0,1,\ldots,L-1;j=1,2,\ldots,t-1\}$, 

$M_3=\{y_iy_{i+j}~|~i=0,1,\ldots,L-1;j=1,2,\ldots,t-1\}$,  

$M_4=\{y_iy_{i-j}~|~i=0,1,\ldots,L-1;j=1,2,\ldots,t-1\}$, 

$M_5=\{x_iy_i~|~i=0,1,\ldots,L-1\}$, 

$M_6=\{x_iy_{i+1}~|~i=0,1,\ldots,L-1\}$,
 
$M_7=\{x_iy_{i+2}~|~i=k-2,k-1,\ldots,L-3\}$ and 

$M_8=\{x_{L-2}y_{k-2},x_{L-1}y_{k-1}\}$,\\
where the subscripts are taken modulo $L$.
 
Let $A=\{x_0,\ldots,x_{k-3}\}$ and $B=\{y_0,\ldots,y_{k-3}\}$ be two outer sets of $D_{k,r}$. 

\begin{lemma}\label{SLCLem}
If $D_{k,r}$ is placed between the vertices $u$ and $u'$ in a triangular terminal graph, then $D_{k,r}$ has the SLC property.
\end{lemma}
\begin{proof}
Observe that $D_{k,r}[V_1]\cong D_{k,r}[V_2]$ and that
$t=\lfloor k/2\rfloor\ge 5$. For each $i\in\{1,2\}$, let
$F_i=D_{k,r}[V_i]$. Then $|V(F_i)|=L$ and $F_i$ is $(2t-2)$-regular. Since
$L=2t-1+r$ and $r\le t-1$, we have
\[
2(2t-2)=4t-4=L+2t-r-3\ge L+t-2\ge L+1.
\]
Thus, by Theorem \ref{ErdosThm}, each $F_i$ is Hamilton-connected.  Applying the same argument, we conclude that $F_i-v$ is Hamilton-connected
for every $v\in V(F_i)$.

Since $F_1$ is Hamilton-connected, for each $j\in\{1,2\}$ there exists a
Hamilton path $P_j$ in $F_1$ from $x_j$ to $x_3$. Since $F_2$ is
Hamilton-connected, there exists a Hamilton path $P_3$ in $F_2$ from $y_3$ to
$y_0$. Hence
\[
ux_1P_1x_3y_3P_3y_0u'
\quad\text{and}\quad
ux_2P_2x_3y_3P_3y_0u'
\]
are two spanning $(u,u')$-paths through $D_{k,r}$ with distinct initial edges.
Therefore, $D_{k,r}$ satisfies \textbf{(A1)}. By symmetry, $D_{k,r}$ also
satisfies \textbf{(A2)}.

It remains to verify \textbf{(A3)}. First suppose that $v\in V_1$. By the
definition of $D_{k,r}$, there exist two distinct vertices
$z_1,z_2\in V_2$ such that $vz_1,\,vz_2\in E_{D_{k,r}}(V_1,V_2)$.
Choose an edge $v_1v_2\in E_{D_{k,r}}(V_1,V_2)$ with
$v_1\in V_1\setminus\{v\}$ and
$v_2\in V_2\setminus\{z_1,z_2\}$. Since $F_2$ is Hamilton-connected, for each
$j\in\{1,2\}$ there exists a Hamilton path $P_j'$ in $F_2$ from $z_j$ to
$v_2$. Choose a vertex $v_3\in A\setminus\{v,v_1\}$. Since $F_1-v$ is
Hamilton-connected, there exists a Hamilton path $Q'$ in $F_1-v$ from $v_1$
to $v_3$. Consequently,
\[
vz_1P_1'v_2v_1Q'v_3u
\quad\text{and}\quad
vz_2P_2'v_2v_1Q'v_3u
\]
are two spanning $(v,u)$-paths through $D_{k,r}$ with distinct initial edges.

Now suppose that $v\in V_2$. Choose two distinct neighbors $w_1$ and $w_2$ of
$v$ in $D_{k,r}[V_2]$. Choose an edge
$s_1s_2\in E_{D_{k,r}}(V_1,V_2)$ such that
$s_1\in V_1$ and $s_2\in V_2\setminus\{v,w_1,w_2\}$. Since $F_2-v$ is
Hamilton-connected, for each $j\in\{1,2\}$ there exists a Hamilton path
$P_j''$ in $F_2-v$ from $w_j$ to $s_2$. Choose a vertex
$s_3\in A\setminus\{s_1\}$. Since $F_1$ is Hamilton-connected, there exists a
Hamilton path $Q''$ in $F_1$ from $s_1$ to $s_3$. Hence
\[
vw_1P_1''s_2s_1Q''s_3u
\quad\text{and}\quad
vw_2P_2''s_2s_1Q''s_3u
\]
are two spanning $(v,u)$-paths through $D_{k,r}$ with distinct initial edges.
Therefore, $D_{k,r}$ satisfies \textbf{(A3)}. Hence, $D_{k,r}$ has the SLC property.
\end{proof}

\begin{theorem}\label{OddThm}
For every odd integer $k\ge 11$, integer $p\ge 6$ and even residue $q\in\{8,10,\ldots,k-3\}$, there exists a $k$-regular doubly homogeneously traceable nonhamiltonian graph of order $p(k-1)+q$. 
\end{theorem}

\begin{proof}
Let $q=2r$. Then $4\le r\le \frac{k-3}{2}=t-1$. Starting from the base graph $G_k$ of Liu and Qiao \cite{Liu2026Qiao}, replace
one copy of the ordinary internal graph $W_k$ with $D_{k,r}$ and leave the other
two copies of $W_k$ unchanged. Denote the resulting graph by $G$. By Lemmas
\ref{DHTLem}, \ref{NHLem}, and \ref{SLCLem}, $G$ is a doubly homogeneously traceable nonhamiltonian graph. Moreover, by construction, $G$ is $k$-regular, and a direct count gives $|V(G)|=6(k-1)+q$.

For every integer $p\ge 6$, apply Lemma \ref{BlowLem} repeatedly to vertices
lying in the internal graphs, adding $k-1$ vertices at each step. Starting from
the graph $G$ above and applying this operation $p-6$ times, we obtain a graph
of order
\[
6(k-1)+q+(p-6)(k-1)=p(k-1)+q.
\]
By Lemmas \ref{BlowLem} and \ref{NHLem}, the resulting graph remains
$k$-regular, doubly homogeneously traceable, and non-Hamiltonian. Hence, the proof is complete.
\end{proof}

\begin{corollary}\label{OddCorol}
For every odd integer $k\ge 5$ and every even $n\ge 6(k-1)$, there exists a $k$-regular doubly homogeneously traceable nonhamiltonian graph of order $n$. 
\end{corollary}

\begin{proof}
By Theorems \ref{LiuThm2026} and  \ref{OddThm}, the statement holds. 
\end{proof}

\subsection{Odd degree $k\ge 7$: all even orders $n\ge 6(k-2)$}\label{Odd7}
In view of Corollary \ref{OddCorol}, it remains to realize the three even orders $6(k-1)-c$ with $c\in\{2,4,6\}$. In this subsection, we first introduce a smaller internal graph, denoted by $S^-(k)$.

Let $V'=\{w_0,w_1,\ldots,w_{2k-7}\}$ be the vertex set of $S^-(k)$. If $k>7$, let $J_1=\{w_iw_{i+j}~|~i=0,1,\ldots,2k-7; j=1,2,\ldots,\frac{k-7}{2}\}$ and $J_2=\{w_iw_{i-j}~|~i=0,1,\ldots,2k-7; j=1,2,\ldots,\frac{k-7}{2}\}$ (where $i-j$ and $i+j$ are taken modulo $2k-6$) be two edge sets. If $k=7$, let $J_1=J_2=\emptyset$. Let $J_3=\{w_iw_{i+k-3}~|~i=0,1,\ldots,k-4\}$ and  $J_4=\{w_0w_{\frac{k-5}{2}}\}$ be edge sets. 

Let $R^-$ be a graph with vertex set $V'$ and edge set $\cup_{j=1}^4 J_i$. Define $S^-(k)=\overline{R^-}$. Let $A=\{w_0,w_1,\ldots,w_{k-3}\}$ and $B=\{w_0,w_{\frac{k-5}{2}}\}\cup\{w_{k-2},w_{k-1},\ldots,w_{2k-7}\}$ be two outer sets of $S^-(k)$. By construction, $d_{S^-(k)}(w_0)=d_{S^-(k)}(w_{\frac{k-5}{2}})=k-2$ and $d_{S^-(k)}(w)=k-1$ for every vertex $w\in V'\setminus\{w_0,w_{\frac{k-5}{2}}\}$.

\begin{lemma}\label{OSLCLem}
For every odd $k\ge 7$, suppose that $S^-(k)$ is placed  between $u$ and $u'$ in a triangular terminal graph. Then $S^-(k)$ has the SLC property. 
\end{lemma}
\begin{proof}
 The proof of $S^-(k)$ is Hamilton-connected is easy (by Theorem \ref{ErdosThm}), so we omit the proof. For any
vertex $v\in V(S^-(k))$, we have $\delta(S^-(k)-v)\ge k-3$ and $|V(S^-(k)-v)|=2k-7$. Since $2(k-3)=(2k-7)+1$, Theorem \ref{ErdosThm} implies that $S^-(k)-v$ is also Hamilton-connected.

Since $S^-(k)$ is Hamilton-connected, for each $j\in\{1,2\}$ there exists a
Hamilton path $P_j$ in $S^-(k)$ from $w_j$ to $w_0$. Hence
\[
uw_1P_1w_0u'
\quad\text{and}\quad
uw_2P_2w_0u'
\]
are two spanning $(u,u')$-paths through $S^-(k)$ with distinct initial edges.
Therefore, $S^-(k)$ satisfies \textbf{(A1)}. By symmetry, $S^-(k)$ also
satisfies \textbf{(A2)}.

It remains to verify \textbf{(A3)}. Let $v\in V(S^-(k))$. Since
$\delta(S^-(k))=k-2\ge 5$, we may choose two distinct neighbors
$z_1,z_2$ of $v$ in $S^-(k)$. Since $|A|=k-2\ge 5$, there exists a vertex $z\in A\setminus\{v,z_1,z_2\}.$
By the Hamilton-connectedness of $S^-(k)-v$, for each $j\in\{1,2\}$ there
exists a Hamilton path $P_j'$ in $S^-(k)-v$ from $z_j$ to $z$. Consequently,
\[
vz_1P_1'zu
\quad\text{and}\quad
vz_2P_2'zu
\]
are two spanning $(v,u)$-paths through $S^-(k)$ with distinct initial edges.
Thus, $S^-(k)$ satisfies \textbf{(A3)}.
\end{proof}

We now prove Theorem \ref{MainResult} for odd degrees $k\ge 7$. 

\begin{MainThmproof1}
Note that $|V(G_k)|=6(k-1)$, $|V(W_k)|=2(k-2)$ and $|V(S^-(k))|$ $=|V(W_k)|-2$. 
Let $c\in\{2,4,6\}$. Starting from the base graph $G_k$ of Liu and Qiao \cite{Liu2026Qiao}, replace exactly $\frac{c}{2}$ copies of the ordinary internal graph $W_k$ with copies of $S^-(k)$, leaving the remaining internal graphs unchanged.
Denote the resulting graph by $G$. Since each such replacement reduces the
order by $2$, we have $|V(G)|=6(k-1)-c.$ Moreover, by construction, $G$ is $k$-regular. By Lemmas \ref{DHTLem},
\ref{NHLem}, and \ref{OSLCLem}, $G$ is a doubly homogeneously
traceable nonhamiltonian graph. Combining this with Corollary \ref{OddCorol},
we obtain the desired conclusion.
\end{MainThmproof1}

\section{The even-degree case of Theorem \ref{MainResult}}
In Subsection \ref{Even10}, we settle all residue classes not covered by Theorem \ref{LiuThm2026} for every even degree $k$. In Subsection \ref{Even6}, we reduce the order threshold in Corollary \ref{LiuCorol2026} to $6(k-2)$ for every even degree $k$.

\subsection{Even degree $k\ge 10$: all orders $n\ge 6(k-1)$}\label{Even10}
In this subsection, we first introduce an internal graph $F_{k,r}$ to settle the remaining residue classes for every even integer $k\ge 10$. Let $t=\lfloor\frac{k}{2}\rfloor$ and choose an integer $3\le r\le t-1$. Let $L=k-2+r=2t-2+r$. Let $(V_1',V_2')$  be a vertex partition of $V(F_{k,r})$, where $V_1'=\{a_0,\ldots,a_{L-1}\}$ and $V_2'=\{b_0,\ldots,b_{L-1}\}$.  Let $A=\{a_0,\ldots,a_{k-3}\}$ and $B=\{b_0,\ldots,b_{k-3}\}$ be two outer sets of $F_{k,r}$. 

We first construct a graph $G_1$ on $V_1'$ and then define $F_{k,r}[V_1']=\overline{G_1}$. Suppose that $r$ is even. Define $A_1=\{a_ia_{i+j}~|~i=0,1,\ldots,L-1; j=1,2,\ldots,\frac{r}{2}\}$ and $A_2=\{a_ia_{i-j}~|~i=0,1,\ldots,L-1; j=1,2,\ldots,\frac{r}{2}\}$ (where $i-j$ and $i+j$ are taken modulo $L$). Let $G_1$ be a graph with vertex set $V_1'$ and edge set $A_1\cup A_2$. Then $G_1$ is $r$-regular, and hence
$F_{k,r}[V_1']$ is $(k-3)$-regular. Applying the same construction to $V_2'$
yields $F_{k,r}[V_2']$.

Now suppose that $r$ is odd. Let
\begin{equation}
\sigma=\begin{cases}
k-1, & \text{if } r=3;\\
L-1, & \text{if } r\ge 5.
\end{cases}
\end{equation}
Let $E_1'=\{a_ia_{i+ j}~|~i=0,1,\ldots,L-1; j=1,2,\ldots,\frac{r-1}{2}\}$ and $E_2'=\{a_ia_{i- j}~|~i=0,1,\ldots,L-1; j=1,2,\ldots,\frac{r-1}{2}\}$  (where $i-j$ and $i+j$ are taken modulo $L$). Let $H$ be the graph with vertex set $V_1'$ and edge set $E_1'\cup E_2'$, and put
$H'=\overline{H-a_{\sigma}}$. Since $H$ is $(r-1)$-regular, we have
$|V(H')|=L-1$ and $\delta(H')=L-1-r$. We claim that $H'$ contains a perfect
matching. Indeed, since $r\le t-1$, we have
\[
r\le \frac{k-2}{2}=\frac{L-r}{2}\le \frac{L-1}{2}.
\]
It follows that
\[
\delta(H')=L-1-r\ge \frac{L-1}{2}.
\]
By Dirac's theorem \cite{Dirac1952}, $H'$ contains a Hamilton cycle $C'$.
Since $L-1$ is even, $C'$ contains a perfect matching, which is also a perfect
matching of $H'$. Fix such a perfect matching and denote its edge set by $E_3'$.
Define $G_1$ on $V_1'$ by  $E(G_1)=\cup_{i=1}^3 E_i'.$
Then $d_{G_1}(a_{\sigma})=r-1$ and $d_{G_1}(a)=r$
for every $a\in V_1'\setminus\{a_{\sigma}\}$. Since
$F_{k,r}[V_1']=\overline{G_1}$, we obtain
\[
d_{F_{k,r}[V_1']}(a_{\sigma})=k-2
\quad\text{and}\quad
d_{F_{k,r}[V_1']}(a)=k-3
\]
for every $a\in V_1'\setminus\{a_{\sigma}\}$. Applying the same construction to
$V_2'$ yields $F_{k,r}[V_2']$.

It remains to define the edge set $E_{F_{k,r}}(V_1',V_2')$. Let $B_1=\{a_ib_i~|~i=0,1,\ldots,L-1\}$, $B_2=\{a_ib_{i+1}~|~i=0,1,\ldots,L-1\}$ (where $i+1$ is taken modulo $L$) and $I=\{k-2,k-1,\ldots,L-1\}$. If $r$ is even, let $I_0=I$. If $r$ is odd, let $I_0=I\setminus\{\sigma\}$. Note that there exists a permutation $\tau$ of $I_0$ satisfying $\tau(j)\not\equiv j$ and $\tau(j)\not\equiv j+1$ (mod $L$) for all $j\in I_0$. Define $B_3=\{a_ib_{\tau(i)}~|~$every$~i\in I_0\}$. Finally, set $E_{F_{k,r}}(V_1',V_2')=\cup_{i=1}^3 B_i$.

\begin{lemma}\label{FSLCLem}
If $F_{k,r}$ is placed between $u$ and $u'$ in a triangular terminal graph, then $F_{k,r}$ has the $SLC$ property. 
\end{lemma}
\begin{proof}
For each $i\in\{1,2\}$, set $F_i=F_{k,r}[V_i']$. Then
$|V(F_i)|=L$. If $r$ is even, then $F_i$ is $(k-3)$-regular. If
$r$ is odd, then, by construction, exactly one vertex of $F_i$ has degree
$k-2$, while every other vertex has degree $k-3$. More precisely, in $F_1$
we have $d_{F_1}(a_{\sigma})=k-2$ and $d_{F_1}(a)=k-3$ for every $a\in V_1'\setminus\{a_{\sigma}\}$, and the analogous statement
holds for $F_2$. Hence, in all cases, $\delta(F_i)=k-3$. Recall that $t=\lfloor k/2\rfloor\ge 5$. Since $r\le t-1$, we have $r\le \frac{k-2}{2}=\frac{L-r}{2}\le \frac{L-3}{2}$. It follows that
\[
2(k-3)=2(L-r-1)\ge L+1.
\]
Thus, by Theorem \ref{ErdosThm}, each $F_i$ is Hamilton-connected.

We next show that $F_i-v$ is Hamilton-connected for every
$v\in V(F_i)$. Since $k\ge 10$, if $r=3$, then $r\le \frac{k-2}{2}-1\le \frac{L-5}{2}.$ If $r\ge 4$, then $r\le \frac{k-2}{2}=\frac{L-r}{2}\le \frac{L-4}{2}.$ Moreover, $\delta(F_i-v)\ge k-4=L-r-2.$  Therefore,
\[
2\delta(F_i-v)\ge 2(L-r-2)\ge L.
\]
Since $|V(F_i-v)|=L-1$, Theorem \ref{ErdosThm} implies that $F_i-v$ is
Hamilton-connected.

Applying the same argument as in the proof of Lemma \ref{SLCLem}, we conclude
that $F_{k,r}$ has the \emph{SLC} property.
\end{proof}

\begin{theorem}\label{EvenThm}
For every even integer $k\ge 10$, integer $p\ge 6$ and residue $q\in\{7,8,\ldots,k-2\}$, there exists a $k$-regular doubly homogeneously traceable nonhamiltonian graph of order $p(k-1)+q$. 
\end{theorem}

\begin{proof}
Suppose first that $q$ is even. Let $q=2r$. Then $4\le r\le \frac{k-2}{2}=t-1$. Starting from the base graph $G_k$ of Liu and Qiao \cite{Liu2026Qiao}, replace one copy of the ordinary internal graph $W_k$ with $F_{k,r}$ and leave the
other two copies of $W_k$ unchanged. Denote the resulting graph by $G$. By
Lemmas \ref{DHTLem}, \ref{NHLem}, and \ref{FSLCLem}, we have $G$ is a doubly homogeneously traceable nonhamiltonian graph. Moreover, by construction, $G$ is
$k$-regular, and a direct count gives $|V(G)|=6(k-1)+q.$

Now suppose that $q$ is odd. Let $q=2r+1$. Then $3\le r\le \frac{k-4}{2}=t-2$. Starting from the base graph $G_k$ of Liu and Qiao \cite{Liu2026Qiao}, replace
two copies of the ordinary internal graph $W_k$ with $F_{k,r}$ and $W_k'$,
respectively, and leave the remaining copy of $W_k$ unchanged. Denote the
resulting graph by $G'$. Suppose that $W_k'$ is placed between $u_1$ and
$u_1'$ as an internal graph. Liu and Qiao \cite{Liu2026Qiao} proved that
$W_k'$ has the \emph{SLC} property. Hence, by Lemmas \ref{DHTLem},
\ref{NHLem}, and \ref{FSLCLem}, we have $G'$ is a doubly homogeneously traceable nonhamiltonian graph. Moreover, by construction, $G'$ is $k$-regular,
and a direct count gives $|V(G')|=6(k-1)+q.$

Finally, for every integer $p\ge 6$, apply Lemma \ref{BlowLem} repeatedly to
vertices lying in the internal graphs, adding $k-1$ vertices at each step.
Starting from $G$ in the even case and from $G'$ in the odd case, and applying
this operation $p-6$ times, we obtain a graph of order
\[
6(k-1)+q+(p-6)(k-1)=p(k-1)+q.
\]
By Lemmas \ref{BlowLem} and \ref{NHLem}, the resulting graph remains
$k$-regular, doubly homogeneously traceable, and nonhamiltonian. Therefore,
there exists a $k$-regular doubly homogeneously traceable nonhamiltonian graph
of order $p(k-1)+q$.
\end{proof}

\begin{corollary}\label{EvenCorol}
For every even integer $k\ge 6$ and every $n\ge 6(k-1)$, there exists a $k$-regular doubly homogeneously traceable nonhamiltonian graph of order $n$. 
\end{corollary}

\begin{proof}
By Theorems \ref{LiuThm2026} and  \ref{EvenThm}, the statement holds. 
\end{proof}

\subsection{Even degree $k\ge 6$: all orders $n\ge 6(k-2)$}\label{Even6}
In view of Corollary \ref{EvenCorol}, it remains to realize the orders $6(k-1)-c$ for $1\le c\le 6$. In this subsection, we introduce two smaller
internal graphs, denoted by $S_1(k)$ and $S_2(k)$.

Let $N=\{v_0,v_1,\ldots,v_{2k-6}\}$ be the vertex set of $S_1(k)$. Let $J_1=\{v_iv_{i+j}~|~i=0,1,\ldots,2k-6; j=1,2,\ldots,\frac{k-4}{2}\}$,  $J_2=\{v_iv_{i- j}~|~i=0,1,\ldots,2k-6; j=1,2,\ldots,\frac{k-4}{2}\}$ (where $i-j$ and $i+j$ are taken modulo $2k-5$) and $J_3=\{v_{2i-1}v_{2i}~|~i=1,2,\ldots,k-3\}$. Let $R_1$ be a graph with vertex set $N$ and edge set $(J_1\cup J_2)\setminus J_3$. Define $S_1(k)=\overline{R_1}$. Let $A_1=\{v_0,v_1,\ldots,v_{k-3}\}$ and $B_1=\{v_0\}\cup\{v_{k-2},v_{k-1},\ldots,v_{2k-6}\}$ be two outer sets of $S_1(k)$. Note that $d_{S_1(k)}(v_0)=k-2$ and $d_{S_1(k)}(v)=k-1$ for every vertex $v\in N\setminus\{v_0\}$.

 Let $M=\{z_0,z_1,\ldots,z_{2k-7}\}$ be the vertex set of $S_2(k)$. If $k>6$, let $D_1=\{z_iz_{i+j}~|~i=0,1,\ldots,2k-7; j=1,2,\ldots,\frac{k-6}{2}\}$ and  $D_2=\{z_iz_{i-j}~|~i=0,1,\ldots,2k-7; j=1,2,\ldots,\frac{k-6}{2}\}$ (where $i-j$ and $i+j$ are taken modulo $2k-6$). If $k=6$, let $D_1=D_2=\emptyset$. Let $D_3=\{z_0z_{k-3}\}$. Let $R_2$ be a graph with vertex set $M$ and edge set $\cup_{i=1}^3D_i$. Define $S_2(k)=\overline{R_2}$. Let $A_2=\{z_0,z_1,\ldots,z_{k-3}\}$ and $B_2=\{z_0\}\cup\{z_{k-3},z_{k-2},\ldots,z_{2k-7}\}$ be two outer sets of $S_2(k)$. Note that $d_{S_2(k)}(z_0)=d_{S_2(k)}(z_{k-3})=k-2$ and $d_{S_2(k)}(z)=k-1$ for every vertex $z\in M\setminus\{z_0,z_{k-3}\}$.

\begin{lemma}\label{ESLCLem}
For every even integer $k\ge 6$ and every $F\in\{S_1(k),S_2(k)\}$, if $F$ is placed between the vertices $u$ and $u'$ as an internal graph in a triangular terminal graph, then $F$ has the SLC property.
\end{lemma}
\begin{proof}
 The proof of $F$ is Hamilton-connected is easy (by Theorem \ref{ErdosThm}), so we omit the proof. Let $x$ be any vertex of $V(F)$. We claim that $F-x$ is also Hamilton-connected. If $F=S_1(k)$, then $(k-2)+(k-1)-2\ge (2k-6)+1$. If $F=S_2(k)$, then $2(k-2-1)\ge (2k-7)+1$. By Theorem \ref{ErdosThm}, $F-x$ is Hamilton-connected. 

First suppose that $F=S_1(k)$. Since $F$ is Hamilton-connected, for each
$j\in\{1,2\}$ there exists a Hamilton path $P_j$ in $F$ from $v_j$ to $v_0$.
Hence
\[
uv_1P_1v_0u'
\quad\text{and}\quad
uv_2P_2v_0u'
\]
are two spanning $(u,u')$-paths through $F$ with distinct initial edges. Thus,
$F$ satisfies \textbf{(A1)}. If $F=S_2(k)$, the same argument shows that $F$
satisfies \textbf{(A1)}. By symmetry, $F$ also satisfies \textbf{(A2)}.

It remains to verify \textbf{(A3)}. Let $v\in V(F)$. First suppose that
$F=S_1(k)$. Choose two distinct neighbors $w_1,w_2\in N_F(v)$. Since
$|A_1|\ge 4$, there exists a vertex $w\in A_1\setminus\{v,w_1,w_2\}.$ As $F-v$ is Hamilton-connected, for each $j\in\{1,2\}$ there exists a
Hamilton path $P_j'$ in $F-v$ from $w_j$ to $w$. Consequently,
\[
vw_1P_1'wu
\quad\text{and}\quad
vw_2P_2'wu
\]
are two spanning $(v,u)$-paths through $F$ with distinct initial edges.
Therefore, $F$ satisfies \textbf{(A3)}. If $F=S_2(k)$, the same argument shows
that $F$ satisfies \textbf{(A3)}.
\end{proof}

We now prove Theorem \ref{MainResult} for even degrees $k\ge 6$.

\begin{MainThmproof2}
Note that $|V(G_k)|=6(k-1)$, $|V(W_k)|=2(k-2)$, $|V(S_1(k))|= |V(W_k)|-1$ and $|V(S_2(k))|= |V(W_k)|-2$. For each $c\in\{1,\ldots,6\}$, starting from the base graph $G_k$ of Liu and
Qiao \cite{Liu2026Qiao}, replace suitable copies of the ordinary internal graph
$W_k$ with copies of $S_1(k)$ or $S_2(k)$ so that the total reduction in order
is $c$. Denote the resulting graph by $G$. Then $|V(G)|=6(k-1)-c.$ Moreover, by construction, $G$ is $k$-regular. By Lemmas \ref{DHTLem},
\ref{NHLem}, and \ref{ESLCLem},  $G$ is a doubly homogeneously traceable nonhamiltonian graph.  Combining this with Corollary \ref{EvenCorol},
we obtain the desired conclusion.
\end{MainThmproof2}

\section*{Declarations}
The authors declare that they have no known competing financial interests or personal relationships that could have appeared to influence the work reported in this paper.
\section*{\bf\Large Availability of Data and Materials}  
Not applicable.
\section*{Acknowledgments}
This research is supported by National Key R\&D Program of China under grant number 2024YFA1013900, NSFC under grant numbers 12471327 and 12401454, Natural Science Foundation of Fujian Province under grant number 2024J01875, Science-Technology Foundation of Putian University under grant number 2023059.


\end{document}